\documentclass[leqno,CJK]{siamltex704}
\usepackage{amssymb,amsmath,graphicx,amscd,mathrsfs}
\usepackage{color,xcolor,amsmath}
\usepackage{amsmath}
\usepackage{graphicx}
\usepackage{mathrsfs}
\usepackage{float}
\usepackage{amsfonts,amssymb}
\usepackage{dsfont}
\usepackage{pifont}
\usepackage{hyperref}
\usepackage{multirow}
\usepackage{placeins}
\numberwithin{equation}{section}
\def\3bar{{|\hspace{-.02in}|\hspace{-.02in}|}}
\def\E{{\mathcal{E}}}
\def\T{{\mathcal{T}}}

\def\b0{\boldsymbol{0}}
\def\sumT{\sum_{T\in\mathcal{T}_h}}     %new
\def\bf{{\mathbf{f}}}

\newtheorem{algorithm1}{Weak Galerkin Algorithm}

 \newcommand{\eps}{\varepsilon}

 \newcommand{\Real}{\mathbb{R}}

\allowdisplaybreaks % For long formula

\begin{document}

\title{A weak Galerkin finite element method for solving the asymptotic lower bound of Maxwell eigenvalue problem}

\author{
	Shusheng Li\thanks{School of Mathematics, Jilin University, Changchun, 130012,
		China (ssli22@mails.jlu.edu.cn).}
		\and	
	Qilong Zhai\thanks{School of Mathematics, Jilin University, Changchun, 130012,
		China (zhaiql@jlu.edu.cn).}
}

%\author{
%Qilong Zhai\thanks{Department of Mathematics, Jilin University, Changchun, 130012,
%China (diql15@mails.jlu.edu.cn).}
%\and
%Hehu Xie\thanks{LSEC, ICMSEC, Academy of Mathematics and Systems Science,
%Chinese Academy of Sciences, Beijing 100190, P.R. China,  and School of Mathematical Sciences, University
%of Chinese Academy of Sciences, Beijing, 100049, China (hhxie@lsec.cc.ac.cn). The research
%of this author was supported by Science Challenge Project (No. TZ2016002),
%National Natural Science Foundations of China (NSFC 11771434, 91330202, 11371026, 11001259),
%the National Center for Mathematics and Interdisciplinary Science, CAS.}
%\and Ran Zhang\thanks{The research of this author was supported in part by China Natural National Science Foundation (91630201, U1530116,11471141, 11726102), and by the Program for Cheung Kong Scholars of Ministry of Education of China, Key Laboratory of Symbolic Computation and Knowledge Engineering of Ministry of Education, Jilin University, Changchun, 130012, P.R. China. } \and
%Zhimin Zhang\thanks{Beijing Computational Science Research Center, Beijing,
%100193, China (zmzhang@csrc.ac.cn); Department of Mathematics, Wayne State University,
%Detroit, MI 48202 (zzhang@math.wayne.edu).
%The research of this author was supported in part by the National
%Natural Science Foundation of China (NSFC 11471031, 91430216)
%and the U.S. National Science Foundation (DMS--1419040).}
%}
\maketitle
%--------------------------------------------------------------------------------------------
\begin{abstract}
In this paper, we propose a weak Galerkin (WG) finite element method for the Maxwell
eigenvalue problem. By restricting subspaces, we transform the mixed form of Maxwell eigenvalue problem into simple elliptic equation. Then we give the WG numerical scheme for the Maxwell eigenvalue problem. Furthermore, we obtain the optimal error estimates of arbitrarily high convergence order and prove the lower bound property of numerical solutions for eigenvalues. Numerical experiments show the accuracy of theoretical analysis and the property of lower bound.
\end{abstract}



\section{Introduction}
The computation of eigenvalues in Maxwell equations is crucial in computational electromagnetics. Since Silvester introduced the nodal-based finite element method to electrical engineering \cite{1969Finite}, it has demonstrated several advantages for solving homogeneous waveguide problems. However, when applied to inhomogeneous waveguide problems, this method can yield spurious modes with nonzero eigenvalues. Considerable effort has been dedicated to reducing or eliminating these unwanted numerical artifacts caused by spurious modes \cite{cendes1986development}. Rahman and Winkler noted that spurious modes fail to satisfy the zero divergence condition on the electric or magnetic field, and they proposed a penalty function to enforce this condition \cite{kobelansky1986eliminating, rahman1984penalty, winkler1984elimination}. Unfortunately, this approach does not completely eliminate the spurious modes and requires users to select an appropriate penalty function parameter.

Another approach to eliminating spurious modes in finite element waveguide problems involves selecting appropriate finite element approximation functions \cite{cendes1986development, MR592160, MR864305}. It is well-established that using $H(\rm{curl; \Omega})$-conforming basis functions (also known as edge elements) for the electric field \cite{MR592160, MR864305} ensure the continuity of tangential field components across interfaces between different media, while allowing normal field components to jump across such interfaces. With the edge element method, spurious modes with nonzero eigenvalues were eliminated. However, the number of spurious modes with zero eigenvalues equal the number of nodal points inside the computational domain, due to the violation of Gauss's law \cite{MR2652780}. Therefore, to remove these zero eigenvalues, it is necessary to enforce the divergence-free property (Gauss's law) of the eigenfunction in addition to using the proper finite element space. One effective approach for computing eigenvalues in waveguide problems is to employ 
$H(\rm{curl; \Omega})$-conforming basis functions to approximate the electric field, while imposing the divergence-free condition using a Lagrange multiplier, as suggested by Kikuchi \cite{MR912525}.

Meanwhile, higher-order methods like the spectral element method \cite{1573841, PhysRevE.79.026705} have been proposed to solve electromagnetic eigenvalue problems. Despite their high convergence rates, these methods also suffer from the presence of spurious zero eigenvalues. In \cite{MR3371510}, Liu et al. introduced the mixed spectral element method, which incorporates the divergence-free equation from Kikuchi's scheme into the spectral element method.

There are many methods to study the Maxwell eigenvalue problems. Based on a Helmholtz decomposition of the error, Boffi et al. \cite{MR3712172} presented an \textit{a posteriori} estimator of the error in the $L^2$-norm for the numerical approximation of the Maxwell eigenvalue problems using N${\rm \acute{e}}$d${\rm \acute{e}}$lec finite elements. Furthermore, based on edge finite elements, Boffi and Gastaldi \cite{MR3916956} proved the optimal convergence of a adaptive scheme. In \cite{MR3918688}, Boffi et al. introduced a residual error indicator for the N${\rm \acute{e}}$d${\rm \acute{e}}$lec finite element approximation for Maxwell eigenvalue problems. In addition, based on a domain decomposition method, Liang and Xu \cite{MR4566815} proposed a two-level preconditioned Helmholtz subspace iterative method for solving algebraic eigenvalue problems arising from edge element approximations of Maxwell eigenvalue problems. In certain settings, Boffi et al. \cite{MR4568428} prove the convergence of the discrete eigenvalues using Lagrange finite elements.

The outline of this paper is as follows. In Section 2, we introduce
the Maxwell eigenvalue problem and three equivalent variational forms are given.
In Section 3, we introduce the weak Galerkin (WG) finite element method for the Maxwell eigenvalue problem, then propose the WG numerical scheme. 
The error estimates of eigenfunctions and eigenvalues are analyzed in Section 4. Section 5 is devoted to prove the lower bounda of eigenvalues.
Numerical examples are presented to our theoretical analysis in Section 6.

%======================================================================================
\section{The Maxwell eigenvalue problem}
In this paper, for simplicity, we consider the following Maxwell eigenvalue problem
\begin{equation}\label{eig problem}
\left\{
\begin{aligned}
\nabla \times \mu_r^{-1} \nabla \times \mathbf{u} &= k^2 \varepsilon_r \mathbf{u},&\quad &\text{in }\Omega,&\\
\nabla \cdot \varepsilon_r \mathbf{u} &= 0,&\quad &\text{in }\Omega,&\\
\mathbf{u}\times \mathbf{n} &=0,&\quad &\text{on }\partial\Omega,&
\end{aligned}
\right.
\end{equation}
where $\Omega$ is a polygon region in $\Real^d$ $(d=2,3)$, and $\varepsilon_r$ and $\mu_r$ denote the relative dielectric permittivity and magnetic permeability tensors, respectively, and $\mathbf{n}$ is the
outward unit normal. For convenience, we consider the case where $\varepsilon_r$ and $\mu_r$ are constants in this paper.

The standard Sobolev space notation are also used in this paper.
 Let $D$ be any open bounded
domain with Lipschitz continuous boundary in $\mathbb{R}^d$ $(d=2, 3)$.
We use the standard definition for the Sobolev space $H^s(D)$ and
their associated inner products $(\cdot, \cdot)_{s, D}$, norms
$\|\cdot\|_{s, D}$, and seminorms $|\cdot|_{s, D}$ for any $s\ge 0$.
When $D=\Omega$, we shall drop the
subscript $D$ in the norm and in the inner product notation. Next, we define the curl operators in 2D:
$$
\nabla \times \mathbf{v}=\partial_x v_2-\partial_y v_1,\quad \nabla \times w=(\partial_y w, -\partial_x w)^T.$$
In addition, we define some Sobolev space used in this paper.
\begin{align*}
	&H({\rm curl}; \Omega)=\left\{\mathbf{v}\in [L^2(\Omega)]^d : \nabla \times \mathbf{v} \in [L^2(\Omega)]^{2d-3}\right\},\\
	&H_0({\rm curl}; \Omega)=\{\mathbf{v}\in H({\rm curl}; \Omega) : \mathbf{v}\times \mathbf{n}=0 \text{ on }\partial\Omega\},\\
	&H_0({\rm curl^0}; \Omega)=\{\mathbf{v}\in H_0({\rm curl}; \Omega) : \nabla \times \mathbf{v} =0\},\\
	&H({\rm div}; \Omega)=\left\{\mathbf{v}\in [L^2(\Omega)]^d : \nabla \cdot \mathbf{v} \in L^2(\Omega)\right\},\\
	&H({\rm div^0}; \Omega)=\{\mathbf{v}\in H({\rm div}; \Omega) : \nabla \cdot \mathbf{v}=0\}.
\end{align*}

Then we introduce the following Helmholtz decomposition:
$$H_0({\rm curl}; \Omega)=U_0\oplus\nabla H_0^1(\Omega),$$
where $U_0=\{\mathbf{v}\in H_0({\rm curl}; \Omega) : (\mathbf{v}, \nabla p)=0, \forall p\in H_0^1(\Omega)\}=\{\mathbf{v}\in H_0({\rm curl}; \Omega) : \nabla \cdot \mathbf{v}=0\}$ and $\nabla H_0^1(\Omega)=H_0({\rm curl^0}; \Omega)$.

The corresponding variational form of (\ref{eig problem}) is: Find $\lambda\in \mathbb{R}$ and $\mathbf{u}\in H_0({\rm curl}; \Omega) \setminus \{0\}$ such that
\begin{equation}\label{form 1}
	\left\{
	\begin{aligned}
		(\mu_r^{-1} \nabla \times \mathbf{u}, \nabla \times \mathbf{v}) &= \lambda (\varepsilon_r \mathbf{u}, \mathbf{v}),& \quad &\forall \mathbf{v}\in H_0({\rm curl}; \Omega),&\\
		(\varepsilon_r \mathbf{u}, \nabla p)&= 0,& \quad &\forall p \in H_0^1(\Omega).&
	\end{aligned}
	\right.
\end{equation}

Here we introduce a mixed form \cite{MR912525} that is equivalent to (\ref{form 1}): Find $\lambda\in \mathbb{R}$ and $\mathbf{u}\in H_0({\rm curl}; \Omega) \setminus \{0\}$, $p\in H_0^1(\Omega)$ such that
\begin{equation}
	\label{form 2}
	\left\{
	\begin{aligned}
		(\mu_r^{-1}\nabla \times \mathbf{u},\nabla \times \mathbf{v})+(\nabla p,\varepsilon_r \mathbf{v})&=\lambda(\varepsilon_r \mathbf{u},\mathbf{v})&\quad &\forall \mathbf{v}\in H_0({\rm curl}; \Omega),&\\ 
		(\varepsilon_r \mathbf{u},\nabla q)&=0 &\quad &\forall q\in H_0^1(\Omega).&
	\end{aligned}
	\right.
\end{equation}

Using the Helmholtz decomposition, we have the following equivalent form: Find $\lambda \in \mathbb{R}$ and $\mathbf{u}\in U_0 \setminus \{0\}$ such that
\begin{equation}
	\label{form 3}
	(\mu_r^{-1}\nabla \times \mathbf{u},\nabla \times \mathbf{v})=\lambda(\varepsilon_r\mathbf{u},\mathbf{v}),\quad \forall \mathbf{v}\in U_0.
\end{equation}

In this paper, we only consider form (\ref{form 3}). Therefore, we give the variational form of Maxwell eigenvalue problem (\ref{eig problem}): Find $\lambda \in \mathbb{R}$ and $\mathbf{u} \in U_0\setminus \{0\}$ such that
\begin{equation}
	\label{weak form}
	a(\mathbf{u},\mathbf{v})=\lambda b(\mathbf{u},\mathbf{v}),\quad \forall \mathbf{v}\in U_0,
\end{equation}
where
\begin{align*}
	a(\mathbf{u},\mathbf{v})=&(\mu_r^{-1}\nabla \times \mathbf{u}, \nabla \times \mathbf{v}),\\
	b(\mathbf{u},\mathbf{v})=&(\varepsilon_r\mathbf{u},\mathbf{v}).
\end{align*}

\section{The weak Galerkin finite element method}
In this section, the weak Galerkin finite element numerical scheme of Maxwell eigenvalue problem is given.

Let $\T_h$ be a partition of the domain $\Omega$, and the elements
in $\T_h$ are polygons satisfying the regular assumptions specified
in \cite{MR3223326}. Denote by $\E_h$ the edges in $\T_h$, and by
$\E_h^0$ the interior edges $\E_h\backslash \partial\Omega$. For
each element $T\in\T_h$, $h_T$ represents the diameter of $T$, and
$h=\max_{T\in\T_h} h_T$ denotes the mesh size. Let $C$ be a positive constant which is
independent of the mesh size and $a \lesssim b$ denote $a\leq Cb$.

Now we introduce the WG space for the eigenvalue problem (\ref{eig problem}).
For a given integer $k\ge 1$, define the WG finite element space
\begin{align*}
	&V_h=\{\mathbf{v}_h=\{\mathbf{v}_0,\mathbf{v}_b\} : \mathbf{v}_0|_T\in [P_k(T)]^d, \mathbf{v}_b|_e \in [P_k(e)]^d, e \in \mathcal {E}_h, \mathbf{v}_b \times \mathbf{n}=0 \text{ on } \partial \Omega\},\\
	&W_h=\{w=\{w_0,w_b\} : w_0|_T \in P_{k-1}(T),w_b|_e \in P_k(e), e \in \mathcal {E}_h, w_b=0 \text{ on } \partial \Omega\}.
\end{align*}

Next, we define the following weak curl operator.
\begin{definition}
	For any $\mathbf{v}_h\in V_h$, its weak curl is defined as the polynomial $\nabla_w \times \mathbf{v}_h|_T $ is the unique polynomial in
	$[P_{k-1}(T)]^{2d-3}$ satisfying
	\begin{eqnarray}
		\label{weak curl}
		(\nabla_w \times \mathbf{v}_h, \mathbf{q})_T=(\mathbf{v}_0,\nabla \times \mathbf{q})_T-\langle \mathbf{v}_b\times \mathbf{n},\mathbf{q}
		\rangle_{\partial T},\quad\forall\mathbf{q}\in [P_{k-1}(T)]^{2d-3}.
	\end{eqnarray}
\end{definition}
In addition, we define the following weak gradient operator.
\begin{definition}
	For any $p_h\in W_h$, its weak gradient is defined as the polynomial $\nabla_w p_h|_T $ is the unique polynomial in
	$[P_{k}(T)]^{d}$ satisfying
	\begin{eqnarray}
		\label{weak gradient}
		(\nabla_w p_h, \phi)_T=-(p_0,\nabla \cdot \phi)_T+\langle p_b,\phi\cdot \mathbf{n}
		\rangle_{\partial T},\quad\forall\phi\in [P_{k}(T)]^{d}.
	\end{eqnarray}
\end{definition}

Then we define the subspace $V_h^0$ of $V_h$:
$$
V_h^0=\{\mathbf{v}_h\in V_h : (\mathbf{v}_0, \nabla_w q)=0, \forall q\in W_h\}.
$$
For any $\mathbf{u}_h, \mathbf{v}_h\in V_h^0$, define the following bilinear forms:
\begin{align*}
	&s(\mathbf{u}_h, \mathbf{v}_h)=\gamma(h) \sum\limits_{T\in \mathcal{T}_{h}}h_T^{-1}\l\langle (\mathbf{u}_0-\mathbf{u}_b)\times \mathbf{n},(\mathbf{v}_0-\mathbf{v}_b)\times \mathbf{n}\rangle_{\partial T},\\
	&a_w(\mathbf{u}_h,\mathbf{v}_h)=\sum\limits_{T\in \mathcal{T}_{h}}(\mu_r^{-1}\nabla_w \times \mathbf{u}_h, \nabla_w \times \mathbf{v}_h)_T + s(\mathbf{u}_h, \mathbf{v}_h)\\
	&b_w(\mathbf{u}_h,\mathbf{v}_h)=(\varepsilon_r\mathbf{u}_0,\mathbf{v}_0),
\end{align*}
where $\gamma(h)=h^{\varepsilon}, 0<\varepsilon<1,\text{ or }\gamma(h)=-\frac{1}{\log(h)}$.

For further error analysis, we introduce some projection operators used
in this paper. Let $\mathbf{Q}_0$ be the $L^2$ projection from $[L^2(T)]^d$
onto $[P_k(T)]^d$, $Q_0$ be the $L^2$ projection from $L^2(T)$
onto $P_{k-1}(T)$.
Let $\mathbf{Q}_b$ denote the $L^2$ projection from $[L^2(e)]^d$ onto
$[P_{k}(e)]^d$, and $Q_b$ denote the $L^2$ projection from $L^2(T)$ onto $P_{k}(T)$. In addition, let $\mathbb{Q}_h$ be the $L^2$ projection from $[L^2(T)]^{2d-3}$
onto $[P_k(T)]^{2d-3}$.
Combining $\mathbf{Q}_0$ and $\mathbf{Q}_b$
together, we can define $\mathbf{Q}_h=\{\mathbf{Q}_0,\mathbf{Q}_b\}$. Analogously, denote $Q_h=\{Q_0,Q_b\}$.

The following commutativity properties are crucial in error analysis.
\begin{lemma}
	\label{projection}
	\cite[Lemma 5.1]{MR3394450}
	For any $\mathbf{v}\in H({\rm curl}; \Omega)$, we have
	$$\nabla_w \mathbf{Q}_h\mathbf{v}=\mathbb{Q}_h(\nabla \times \mathbf{v}),$$
	and for any $q\in H^1(\Omega)$, there holds
	$$\nabla_w Q_h q=\mathbf{Q}_0(\nabla q).$$
\end{lemma}

Now we give the numerical scheme of (\ref{eig problem}) corresponding to (\ref{form 3}).
\begin{algorithm1}
	Find $(\lambda_h, \mathbf{u}_h)\in\mathbb{R}\times V_h^0$ such that $\|\mathbf{u}_h\|=1$ and
	\begin{eqnarray}\label{WG-scheme}
		a_w(\mathbf{u}_h,\mathbf{v}_h)=\lambda_h b_w(\mathbf{u}_0,\mathbf{u}_0),\quad\forall \mathbf{v}_h\in V_h^0.
	\end{eqnarray}
\end{algorithm1}

Define the sum space $V=U_0+V_h^0$.
Now we define the following semi-norm on $V$. For any $\mathbf{v}\in V$,
\begin{eqnarray*}
	\|\mathbf{v}\|_V^2= \sumT \Big(\mu_r^{-1}\|\nabla \times \mathbf{v}_0\|_T^2+ h_T^{-1}\|
	(\mathbf{v}_0-\mathbf{v}_b)\times\mathbf{n}\|^2_{\partial T}\Big),
\end{eqnarray*}
where $\mathbf{v}_0$ is the internal value and $\mathbf{v}_b$ is the boundary value.

The following Friedrichs inequality \cite{MR3097958} holds on $U_0$.
\begin{lemma}
	\label{Friedrichs inequality}
	Let $\Omega$ be a bounded Lipschitz domain, there exists a positive constant $C$ such that
	$$
	\|\mathbf{u}\|\leq C\|\nabla \times \mathbf{u}\|, \quad \forall \mathbf{u}\in U_0.$$
\end{lemma}

Therefore, from Lemma \ref{Friedrichs inequality}, $\|\cdot\|_V$ is a norm on $U_0$.
\begin{lemma}
	$\|\cdot\|_V$ defines a norm on $V_h^0$.
\end{lemma} 
\begin{proof}
	We prove that $\|\mathbf{v}_h\|_V=0$ if and only if $\mathbf{v}_h=0$.
	If $\mathbf{v}_h=0$, obviously, $\|\mathbf{v}_h\|_V=0$.
	Next, assume that $\|\mathbf{v}_h\|_V=0$. Then we have $\nabla \times \mathbf{v}_0=0$ for any $T\in \mathcal{T}_h$, and $(\mathbf{v}_0-\mathbf{v}_b)\times\mathbf{n}=0$ for any $e\in \mathcal{E}_h$.
	Notice that $\mathbf{v}_h \in V_h^0$, then for any $q_h\in W_h$, there holds
	\begin{align*}
		b_w(\mathbf{v}_0, q_h)&=\sumT(\varepsilon_r\mathbf{v}_0, \nabla_wq_h)_T\\
		&=-\sumT(\nabla \cdot (\varepsilon_r \mathbf{v}_0), q_0)_T+\sumT \langle \varepsilon_r \mathbf{v}_0\cdot \mathbf{n}, q_b \rangle_{\partial T}\\
		&=0.
	\end{align*}
	Specially, let $q_0=\nabla \cdot \mathbf{v}_0$ and $q_b=0$, then we have $\nabla \cdot \mathbf{v}_0=0$ for any $T\in \mathcal{T}_h$. Let $q_0=0$ and $q_b$ is the jump of $\mathbf{v}_0\cdot \mathbf{n}$ on $e$, then we can get that $\mathbf{v}_0\cdot \mathbf{n}$ is continuous on $e$. 
	Notice that $\nabla \times \mathbf{v}_0= 0$, then there exists a potential function $\phi$ such that $\mathbf{v}_0 =\nabla \phi$ in $\Omega$. In addition,
	$\nabla \cdot \mathbf{v}_0= 0$ and $\mathbf{v}_0 \cdot \mathbf{n}$ is continuous, there holds $\phi = 0$
	in $\Omega$. Notice that $\mathbf{v}_b \times \mathbf{n}=0$ on 
	$\partial \Omega$, and $(\mathbf{v}_0-\mathbf{v}_b) \times \mathbf{n}=0$ for any $e\in \mathcal{E}_h$,
	Thus, $\mathbf{v}_0\times\mathbf{n}=\nabla \phi=0$ on $\partial \Omega$. Therefore,
	$\phi$ must be a constant on $\partial \Omega$. From the uniqueness of the solution of the Laplace equation, if $\phi$ is simply connected, $\phi$ must be a constant in $\Omega$. 
	Then $\mathbf{v}_0=\nabla \phi=0$ and $\mathbf{v}_b=0$.
\end{proof}

\section{Error analysis}
By limiting to subspaces, we transform the Maxwell eigenvalue problem into an elliptic eigenvalue problem. Therefore, we can use the framework presented in \cite{MR3919912} and just prove Assumption 1 (A1)-Assumption 7 (A7) in Section 2 of \cite{MR3919912}.

First we verify (A1):
$a(\cdot,\cdot)$ and $a_w(\cdot,\cdot)$ are symmetric, and
\begin{align*}
	&a(\mathbf{w},\mathbf{w})\geq \gamma\Vert \mathbf{w} \Vert_V^2,\quad \forall \mathbf{w} \in U_0,\\
	&a_w(\mathbf{v}_h,\mathbf{v}_h)\geq \gamma(h)\Vert \mathbf{v}_h\Vert_V^2,\quad \forall \mathbf{v}_h \in V_h,
\end{align*}
where $\gamma$ and $\gamma(h)$ are two positive constants for given $h$.

Obviously, $a(\cdot, \cdot)$ is coercive and continuous on $U_0$. And $a_w(\cdot,\cdot)$ is continuous on $V_h^0$, we only prove that $a_w(\cdot,\cdot)$ is coercive on $V_h^0$.
\begin{lemma}
	For any $\mathbf{v}_{h} \in V^0_{h}$, we have
	$$a_{w}(\mathbf{v}_{h},\mathbf{v}_{h})\geq C \gamma(h)\Vert \mathbf{v}_{h} \Vert_{V}^{2}.$$
\end{lemma}
\begin{proof}
	It is easy to prove
	$$\sumT h_T^{-1}\|
	(\mathbf{v}_0-\mathbf{v}_b)\times\mathbf{n}\|^2_{\partial T}\leq C\gamma(h)^{-1}a_{w}(\mathbf{v}_{h},\mathbf{v}_{h}).$$
	From integration by parts, (\ref{weak curl}), Young inequality, trace inequality \cite{MR3223326} and inverse inequality \cite{MR3223326}, we have
	\begin{align*}
		&\sumT(\nabla \times \mathbf{v}_0, \nabla \times \mathbf{v}_0)_T\\
		=&\sumT(\nabla \times \nabla \times \mathbf{v}_0, \mathbf{v}_0)_T-\sumT\langle \nabla \times \mathbf{v}_0,  \mathbf{v}_0 \times \mathbf{n} \rangle_{\partial T}\\
		=&\sumT(\nabla_w \times \mathbf{v}_h, \nabla \times \mathbf{v}_0)_T+\sumT\langle \nabla \times \mathbf{v}_0, (\mathbf{v}_b-\mathbf{v}_0)\times \mathbf{n} \rangle_{\partial T}\\
		\leq &C\left( \sumT\|\nabla_w \times \mathbf{v}_h\|_T^2\right)^{1/2}\left(\sumT \|\nabla \times \mathbf{v}_0\|_T^2 \right)^{1/2}\\
		&+C\left( \sumT h_T\|\nabla \times \mathbf{v}_0\|_{\partial T}^2\right)^{1/2}\left(h_T^{-1}\|(\mathbf{v}_b-\mathbf{v}_0)\times \mathbf{n}\|_{\partial T}^2\right)^{1/2}\\
		\leq &C \sumT\|\nabla_w \times \mathbf{v}_h\|_T^2+\frac{1}{4}\sumT \|\nabla \times \mathbf{v}_0\|_T^2\\
		&+C\sumT h_T^{-1}\|(\mathbf{v}_b-\mathbf{v}_0)\times \mathbf{n}\|_{\partial T}^2 + \frac{1}{4}\sumT h_T\| \nabla \times \mathbf{v}_0\|_{\partial T}^2\\
		\leq &C \sumT\|\nabla_w \times \mathbf{v}_h\|_T^2+\frac{1}{4}\sumT \|\nabla \times \mathbf{v}_0\|_T^2\\
		&+C\sumT h_T^{-1}\|(\mathbf{v}_b-\mathbf{v}_0)\times \mathbf{n}\|_{\partial T}^2 + \frac{1}{4}\sumT \| \nabla \times \mathbf{v}_0\|_{T}^2\\
		&\leq \frac{1}{2}\sumT \| \nabla \times \mathbf{v}_0\|_{T}^2+C\gamma(h)^{-1}a_w(\mathbf{v}_h,\mathbf{v}_h),
	\end{align*}
which implies that
$$
\sumT\|\nabla \times \mathbf{v}_0\|_T^2\leq C\gamma(h)^{-1}a_w(\mathbf{v}_h,\mathbf{v}_h).
$$
Combining the two equations above, we have
$$
\Vert \mathbf{v}_{h} \Vert_{V}^{2}\leq C \gamma(h)^{-1}a_{w}(\mathbf{v}_{h},\mathbf{v}_{h}),$$
which completes the proof.
\end{proof}

Before verifying (A2), we define the following two operators $K$ and $K_h$.
\begin{align*}
	&K: L^2(\Omega)\rightarrow U_{0}, \quad a(K\mathbf{f},\mathbf{v})=b(\mathbf{f},\mathbf{v}), \quad \forall \mathbf{v} \in U_{0},\\
	&K_{h}: V^0_{h}+L^2(\Omega)\rightarrow V_{h}^0,\quad a_{w}(K_{h}\mathbf{f}_{h},\mathbf{v}_{h})=b_{w}(\mathbf{f}_{h},\mathbf{v}_{h}), \quad \forall \mathbf{v}_{h}\in V^0_{h}.
\end{align*}

It follows from \cite[Lemma 5.3]{MR3325251} and \cite[Lemma 7.1]{MR3286455} that
$$
\|\mathbf{v}_0\|\leq C\|\mathbf{v}_h\|_V.
$$
Therefore, it is easy to prove that $K$ and $K_{h}$ are well-defined by the Lax-Milgram theorem.

Next we verify (A2): $K$ and $K_h$ are compact.
\begin{lemma}
	$K$ and $K_h$ are compact operators from $U_0$ onto $U_0$ and $V_h^0$ onto $V_h^0$, respectively.
\end{lemma}
\begin{proof}
	As we all know, it follows from \cite[Theorem 4.7]{MR2059447} that $U_0$ is compactly embedded in $L^2(\Omega)$. By the definition of $K$, we know that $K$ is a bounded operator from $L^2(\Omega)$ onto $U_0$. Because bounded operator compound compact operator is still compact operator, $K$ is compact operator from $U_0$ onto $U_0$. For $V_h^0$, it is bounded by Lax-Milgram Theorem, and it is also finite rank, so $K_h$ is compact operator from $V_h^0$ onto $V_h^0$. Then we complete the proof.
\end{proof}

Now we verify (A3): there exists a bounded linear operator $\mathbf{Q}_h: V\rightarrow V_h^0$ satisfying
\begin{align*}
	&\mathbf{Q}_h \mathbf{w}_h=\mathbf{w}_h,\quad \forall \mathbf{w}_h\in V_h^0,\\
	&b(\mathbf{Q}_h \mathbf{f},\mathbf{w}_h)=b(\mathbf{f},\mathbf{w}_h),\quad \forall \mathbf{f}\in V,\mathbf{w}_h\in V_h^0.
\end{align*}

Obviously, for any $\mathbf{w}_{h}\in V_{h}^0$, $\mathbf{Q}_{h}\mathbf{w}_{h}=\mathbf{w}_{h}$, and
$$b(\mathbf{f},\mathbf{w}_{h})=(\varepsilon_r \mathbf{f}_0,\mathbf{w}_0)=(\varepsilon_r \mathbf{Q}_0\mathbf{f}_0, \mathbf{w}_0)=b(\mathbf{Q}_{h}\mathbf{f},\mathbf{w}_{h}).$$
Further we need to prove $\mathbf{Q}_{h}\mathbf{f} \in V_h^0$. For any $q_h\in W_h$, it follows from (\ref{weak gradient}) that
\begin{align*}
	(\varepsilon_r \mathbf{Q}_{h}\mathbf{f}, \nabla_w q_h)&=(\varepsilon_r \mathbf{Q}_{0}\mathbf{f}, \nabla_w q_h)\\
	&=(\varepsilon_r \mathbf{f}, \nabla_w q_h)\\
	&=-(\nabla \cdot (\varepsilon_r \mathbf{f}), q_0)+\sumT \langle \varepsilon_r \mathbf{f}\cdot \mathbf{n}, q_b \rangle_{\partial T}\\
	&=0,
\end{align*}
which uses the fact that $\sum\limits_{T\in \mathcal{T}_{h}}\langle \mathbf{f}\cdot \mathbf{n}, q_b \rangle_{\partial T}=0$. Then (A3) is verified.

Denote
\begin{align*}
	&\delta_{h,\mu}=\sup\limits_{\mathbf{v}\in R(E_{\mu}(K)),\atop \Vert \mathbf{v} \Vert_{V}=1}\Vert \mathbf{v}-Q_{h}\mathbf{v} \Vert_{V},\\
	&\delta_{h,\mu}'=\sup\limits_{\mathbf{v}\in R(E_{\mu}(K)),\atop \Vert \mathbf{v} \Vert_{V}=1}\Vert \mathbf{v}-\mathbf{Q}_{h}\mathbf{v} \Vert,
\end{align*}
where $R(E_{\mu}(K))$ is the space of eigenfunctions corresponding to the eigenvalue $\mu$, and can be found in Section 2 of \cite{MR3919912}.

Let $\Pi_{0}$ be the projection operator from $V$ onto $L^2(\Omega)$ under $b(\cdot,\cdot)$, denote
\begin{align*}
	&e_{h,\mu}=\Vert (K-K_{h}Q_{h})|_{R(E_{\mu}(K)} \Vert_{V},  \\
	&e_{h,\mu}'=\Vert (\Pi_{0}K-\Pi_{0}K_{h}Q_{h})|_{R(E_{\mu}(K))} \Vert_{X}.
\end{align*}

Next we verify (A4) and (A5):
\begin{align*}
	&(A4): e_{h,\mu}\rightarrow 0 ~{\rm as} ~h \rightarrow 0, {\rm and} ~\delta_{h,\mu}\gamma(h)^{-1}\rightarrow 0~{\rm as} ~h \rightarrow 0,\\
	&(A5): e_{h,\mu}'\rightarrow 0 ~{\rm as} ~h \rightarrow 0, {\rm and} ~\delta_{h,\mu}'\gamma(h)^{-1}\rightarrow 0~{\rm as} ~h \rightarrow 0.
\end{align*}

\begin{lemma}
	\label{delta error}
	Assume $R(E_{\mu}(K))\subset H^{k+1}(\Omega)$, there hold the following estimates
	\begin{align*}
		&\delta_{h,\mu}\lesssim h^{k},\\
		&\delta_{h,\mu}'\lesssim h^{k+1}.
	\end{align*}
\end{lemma}
\begin{proof}
	It follows from the trace inequality and the projection inequality \cite{MR3452926} that
	\begin{align*}
		&\Vert \mathbf{u}-\mathbf{Q}_{h}\mathbf{u} \Vert_{V}^{2} \\
		=&\sum\limits_{T\in \mathcal{T}_{h}}\Vert \mu_r^{-1}\nabla\times (\mathbf{u}-\mathbf{Q}_{0}\mathbf{u}) \Vert_{T}^{2}+\sum\limits_{T\in \mathcal{T}_{h}}h_{T}^{-1}\Vert (\mathbf{Q}_{0}\mathbf{u}-\mathbf{Q}_{b}\mathbf{u})\times \mathbf{n} \Vert_{\partial T}^{2}  \\
		\leq&\sum\limits_{T\in \mathcal{T}_{h}}\Vert \mu_r^{-1}\nabla\times (\mathbf{u}-\mathbf{Q}_{0}\mathbf{u}) \Vert_{T}^{2}+\sum\limits_{T\in \mathcal{T}_{h}}h_{T}^{-1}\Vert \mathbf{Q}_{0}\mathbf{u}-\mathbf{u} \Vert_{\partial T}^{2}  \\
		\leq&\sum\limits_{T\in \mathcal{T}_{h}}\Vert \mu_r^{-1}\nabla\times (\mathbf{u}-\mathbf{Q}_{0}\mathbf{u}) \Vert_{T}^{2}
		+\sum\limits_{T\in \mathcal{T}_{h}}C(h_{T}^{-2}\Vert \mathbf{Q}_{0}\mathbf{u}-\mathbf{u} \Vert_{T}^{2}+\Vert \nabla(\mathbf{Q}_{0}\mathbf{u}-\mathbf{u}) \Vert_{T}^{2})  \\
		\lesssim&h^{2k},
	\end{align*}
	and
	$$\Vert \mathbf{u}-\mathbf{Q}_{h}\mathbf{u} \Vert^{2}\lesssim h^{2(k+1)}.$$
	Then we complete the proof.
\end{proof}

For further analysis, we consider the following source problem corresponding the eigenvalue problem (\ref{eig problem}):
\begin{equation}\label{source problem}
	\left\{
	\begin{aligned}
		\nabla \times \mu_r^{-1} \nabla \times \mathbf{u} &=  \varepsilon_r \mathbf{f},&\quad &\text{in }\Omega,&\\
		\nabla \cdot \varepsilon_r \mathbf{u} &= 0,&\quad &\text{in }\Omega,&\\
		\mathbf{u}\times \mathbf{n} &=0,&\quad &\text{on }\partial\Omega,&
	\end{aligned}
	\right.
\end{equation}
where $\mathbf{f}\in H({\rm div^0}; \Omega)$.

The corresponding variational form is: Find $\mathbf{u}\in H_0({\rm curl}; \Omega) $ and $p\in H_0^1(\Omega)$ such that
\begin{equation}
	\label{weak form for source}
	\left\{
	\begin{aligned}
		(\mu_r^{-1}\nabla \times \mathbf{u},\nabla \times \mathbf{v})+(\nabla p,\varepsilon_r \mathbf{v})&=(\varepsilon_r \mathbf{f},\mathbf{v})&\quad &\forall \mathbf{v}\in H_0({\rm curl}; \Omega),&\\ 
		(\varepsilon_r \mathbf{u},\nabla q)&=0 &\quad &\forall q\in H_0^1(\Omega).&
	\end{aligned}
	\right.
\end{equation}
Then we give the WG numerical scheme for (\ref{weak form for source}): Find $\mathbf{u}_h\in V_h$ and $p_h\in W_h$ such that
\begin{equation}
	\label{numerical scheme for source}
	\left\{
	\begin{aligned}
		&a_w(\mathbf{u}_h,\mathbf{v}_h)+(\varepsilon_r \mathbf{v}_0, \nabla_w p_h)=(\varepsilon_r \mathbf{f},\mathbf{v}_0)&\quad &\forall \mathbf{v}_h\in V_h,&\\ 
		&(\varepsilon_r \mathbf{u}_0,\nabla_w q_h)+\sumT h_T\langle p_0-p_b, q_0-q_b\rangle_{\partial T}=0 &\quad &\forall q_h\in W_h.&
	\end{aligned}
	\right.
\end{equation}

\begin{lemma}\cite[Theorem 7.2]{MR3394450}
	Let $(\mathbf{u}; p)\in [H^{k+1}(\Omega)]^d\times[H^1_0(\Omega)\cap H^k(\Omega)]$ be the solution of (\ref{weak form for source}) and $(\mathbf{u}_h; p_h)\in V_h\times W_h$ be the numerical solution of (\ref{numerical scheme for source}), then
	$$\|\mathbf{Q}_h\mathbf{u}-\mathbf{u}_h\|_V\leq C\gamma(h)^{-1}h^k\|\mathbf{u}\|_{k+1}.$$
\end{lemma}
\begin{lemma}\cite[Theorem 8.1]{MR3394450}
	Let $(\mathbf{u}; p)\in [H^{k+1}(\Omega)]^d\times[H^1_0(\Omega)\cap H^k(\Omega)]$ be the solution of (\ref{weak form for source}) and $(\mathbf{u}_h; p_h)\in V_h\times W_h$ be the numerical solution of (\ref{numerical scheme for source}), then
	$$\|\mathbf{Q}_0\mathbf{u}-\mathbf{u}_0\|_V\leq C\gamma(h)^{-1}h^{k+1}\|\mathbf{u}\|_{k+1}.$$
\end{lemma}

By triangle inequality, we have
\begin{align*}
	&\|\mathbf{u}-\mathbf{u}_h\|_V \lesssim \gamma(h)^{-1}h^k,\\
	&\|\mathbf{u}-\mathbf{u}_h\| \lesssim \gamma(h)^{-1}h^{k+1}.
\end{align*}

Notice that $e_{h,\mu}$ and $e_{h,\mu}'$ are the $\|\mathbf{u}-\mathbf{u}_h\|_V$ and $\|\mathbf{u}-\mathbf{u}_h\|$ error estimates of the source problem (\ref{source problem}), respectively. Then the following lemma holds.
\begin{lemma}
	\label{e error}
	Assume $R(E_{\mu}(K))\subset H^{k+1}(\Omega)$, then
	\begin{align*}
		&e_{h,\mu}\lesssim \gamma(h)^{-1}h^{k}, \\
		&e_{h,\mu}'\lesssim \gamma(h)^{-1}h^{k+1}.
	\end{align*}
\end{lemma}

According to \cite[Theorem 2.1, Theorem 2.2]{MR3919912}, the following theorem holds. The details can be seen in \cite[Theorem 7.1]{MR1115240}.

\begin{theorem}
	\label{eigfunction error}
	Let $\lambda$ be an eigenvalue of $(\ref{weak form})$ with multiplicity $m$, and $\{\mathbf{u}_{j}\}_{j=1}^{m}$ denote a basis of the corresponding m-dimensional eigenspace $R(E_{\mu}(K))\subset H^{k+1}(\Omega)$. Similarly, let $\{\lambda_{j,h}\}_{j=1}^{m}$ be the eigenvalues of $(\ref{WG-scheme})$, and $\{\mathbf{u}_{j,h}\}_{j=1}^{m}$ denote a basis of the corresponding eigenspace $R(E_{\mu,h}(K_{h}))$. For any $j=1,\cdots,m$, there exists an eigenfunction $\mathbf{u}_{j}\in R(E_{\mu}(K))$ such that
	\begin{align*}
		&\Vert \mathbf{u}_{j}-\mathbf{u}_{j,h} \Vert_{V}\lesssim \gamma(h)^{-1}h^{k},\\
		&\Vert \mathbf{u}_{j}-\mathbf{u}_{j,h} \Vert\lesssim \gamma(h)^{-1}h^{k+1}.
	\end{align*}
\end{theorem}

Notice that the parameter $\gamma(h)$ is chosen as either $\gamma(h)=h^{\varepsilon} (0< \eps<1)$, or $\gamma(h)=-1/\log(h)$, then the conclusions in Theorem $\ref{eigfunction error}$ can be written as follows
\begin{align*}
	&\Vert \mathbf{u}_{j}-\mathbf{u}_{j,h} \Vert_{V}\lesssim h^{k-\varepsilon},\\
	&\Vert \mathbf{u}_{j}-\mathbf{u}_{j,h} \Vert\lesssim h^{k+1-\varepsilon},
\end{align*}
or
\begin{align*}
	&\Vert \mathbf{u}_{j}-\mathbf{u}_{j,h} \Vert_{V}\lesssim -\log(h)h^{k},\\
	&\Vert \mathbf{u}_{j}-\mathbf{u}_{j,h} \Vert\lesssim -\log(h)h^{k+1}.
\end{align*}

Let $\varepsilon_{h,\mathbf{u}}=a(\mathbf{u},\mathbf{u})-a_{w}(\mathbf{Q}_{h}\mathbf{u},\mathbf{Q}_{h}\mathbf{u})$. Then we verify (A6): $\forall \mu\in\sigma_K$, $\forall \mathbf{u}\in R(E_{\mu}(K))$, $\varepsilon_{h,\mu}\rightarrow0$ as $h\rightarrow0$, where $\sigma_K$ is the spectrum of $K$. 

\begin{lemma}
	\label{eps error}
	Let $\mathbf{u}\in H^{k+1}(\Omega)$, there holds
	$$\vert \varepsilon_{h,\mathbf{u}} \vert \lesssim h^{2k}.$$
\end{lemma}
\begin{proof}
	It follows from Lemma \ref{projection} and projection inequalities \cite[Lemma 6.1]{MR3394450} that
	\begin{align*}
		&\vert \varepsilon_{h,\mathbf{u}} \vert  \\
		=&\left\vert \sum\limits_{T\in \mathcal{T}_{h}}\Vert \nabla \times \mathbf{u} \Vert_{T}^{2}-\sum\limits_{T\in \mathcal{T}_{h}}\Vert \nabla_{w}\times \mathbf{Q}_{h}\mathbf{u} \Vert_{T}^{2}-s(\mathbf{Q}_{h}\mathbf{u},\mathbf{Q}_{h}\mathbf{u}) \right\vert   \\         
		=&\left\vert \sum\limits_{T\in \mathcal{T}_{h}}\Vert \nabla \times \mathbf{u} \Vert_{T}^{2}-\sum\limits_{T\in \mathcal{T}_{h}}\Vert \mathbb{Q}_{h}\nabla \times \mathbf{u} \Vert_{T}^{2}-s(\mathbf{Q}_{h}\mathbf{u},\mathbf{Q}_{h}\mathbf{u}) \right\vert    \\       
		\leq&\sum\limits_{T\in \mathcal{T}_{h}}\Vert \nabla \times \mathbf{u}-\mathbb{Q}_{h}\nabla \times \mathbf{u} \Vert_{T}^{2}+\sum\limits_{T\in \mathcal{T}_{h}}\gamma(h)h_{T}^{-1}\Vert \mathbf{Q}_{0}\mathbf{u}-\mathbf{u} \Vert_{\partial T}^{2}      \\    
		\leq&\sum\limits_{T\in \mathcal{T}_{h}}\Vert \nabla \times \mathbf{u}-\mathbb{Q}_{h}\nabla \times \mathbf{u} \Vert_{T}^{2}
		+\sum\limits_{T\in \mathcal{T}_{h}}C\gamma(h)(h_{T}^{-2}\Vert \mathbf{Q}_{0}\mathbf{u}-\mathbf{u} \Vert_{T}^{2}+\Vert \nabla(\mathbf{Q}_{0}\mathbf{u}-\mathbf{u}) \Vert_{T}^{2})    \\      
		\lesssim& h^{2k},
	\end{align*}
	which completes the proof of this lemma.
\end{proof}

Then the following error estimate of eigenvalues holds.
\begin{theorem}
	\label{eigenvalue error}
	Let $\lambda$ be an eigenvalue of $(\ref{weak form})$ with multiplicity $m$, and $\{\mathbf{u}_{j}\}_{j=1}^{m}$ be a basis of the corresponding m-dimensional eigenspace $R(E_{\mu}(K))\subset H^{k+1}(\Omega)$. Let $\{\lambda_{j,h}\}_{j=1}^{m}$ denote the eigenvalues of $(\ref{WG-scheme})$, and $\{\mathbf{u}_{j,h}\}_{j=1}^{m}$ denote a basis of the corresponding eigenspace $R(E_{\mu,h}(K_{h}))$. When $h$ is sufficiently small, for any $j=1,\cdots,m$, the following estimate holds
	$$\vert \lambda-\lambda_{j,h} \vert \lesssim \gamma(h)^{-2}h^{2k}.$$
\end{theorem}

\begin{proof}
	It follows from \cite[Theorem 2.3]{MR3919912} that
	$$\vert \lambda-\lambda_{j,h} \vert \leq C(\varepsilon_{h,\mathbf{u}_{j}}+e_{h,\lambda^{-1}}^{2}+e_{h,\lambda^{-1}}'^{2}+\delta_{h,\lambda^{-1}}^{2}\gamma(h)^{-2}+
	\delta_{h,\lambda^{-1}}'^{2}\gamma(h)^{-2}).$$
	From Lemmas $\ref{delta error}$, $\ref{e error}$ and $\ref{eps error}$, we have
	\begin{align*}
		&\delta_{h,\lambda^{-1}}\lesssim h^{k}, ~\delta_{h,\lambda^{-1}}'\lesssim h^{k+1},\\
		&e_{h,\lambda^{-1}}\lesssim \gamma(h)^{-1}h^{k}, ~e_{h,\lambda^{-1}}'\lesssim \gamma(h)^{-1}h^{k+1},\\
		&\vert \varepsilon_{h,\mathbf{u}} \vert \lesssim h^{2k},
	\end{align*}
	which implies that $$\vert \lambda-\lambda_{j,h} \vert \lesssim \gamma(h)^{-2}h^{2k}.$$
	Then we complete the proof.
\end{proof}

Notice that the parameter $\gamma(h)$ is chosen as either $\gamma(h)=h^{\varepsilon} (0< \eps<1)$, or $\gamma(h)=-1/\log(h)$, the conclusion in Theorem $\ref{eigenvalue error}$ is as follows
$$\vert \lambda-\lambda_{j,h} \vert \lesssim h^{2k-2\varepsilon},$$
or
$$\vert \lambda-\lambda_{j,h} \vert \lesssim {\log(h)}^2h^{2k}.$$

\section{Asymptotic lower bound}
In this section, we show that the numerical eigenvalues are the lower bound of the exact eigenvalues.

It only need to verify (A7): Suppose $(\lambda,\mathbf{u})$ is an eigenpair of $(\ref{weak form})$, $(\lambda_{h},\mathbf{u}_{h})$ is an eigenpair of $(\ref{WG-scheme})$, there holds $\varepsilon_{h,\mathbf{u}}\geq \lambda_{h}\Vert \mathbf{u}-\mathbf{u}_{h} \Vert^{2}.$

The following estimates are crucial, and the proof can be found in \cite[Theorem 2.1]{MR3120579}.
\begin{lemma}
	\label{projection lower bound}
	The exact eigenfunction $\mathbf{u}$ of the Steklov eigenvalue problem $(\ref{weak form})$ exhibits lower bounds for its convergence rates as follows
	\begin{align*}
	\sum\limits_{T\in \mathcal{T}_{h}}\Vert \nabla \times \mathbf{u}-\mathbb{Q}_{h}\nabla \times \mathbf{u} \Vert_{T}^{2}\gtrsim h^{2k}.
	\end{align*}
\end{lemma}

\begin{lemma}
	Let $(\lambda,\mathbf{u})$ be an eigenpair of $(\ref{weak form})$, $(\lambda_{h},\mathbf{u}_{h})$ be an eigenpair of $(\ref{WG-scheme})$. For sufficiently small $h$, $\gamma(h)\ll 1$, the following estimate holds
	$$\varepsilon_{h,\mathbf{u}}\geq \lambda_{h}\Vert \mathbf{u}-\mathbf{u}_{h} \Vert_{X}^{2}.$$
\end{lemma}
\begin{proof}
	It follows from Lemma $\ref{projection}$ that
	\begin{flalign*}
		&\varepsilon_{h,\mathbf{u}}\\
		=&a(\mathbf{u},\mathbf{u})-a_{w}(\mathbf{Q}_{h}\mathbf{u},\mathbf{Q}_{h}\mathbf{u})      \\
		=&\sum\limits_{T\in \mathcal{T}_{h}}\Vert \nabla\times  \mathbf{u} \Vert_{T}^{2}-\sum\limits_{T\in \mathcal{T}_{h}}\Vert \nabla_{w}\times \mathbf{Q}_{h}\mathbf{u} \Vert_{T}^{2}-s(Q_{h}u,Q_{h}u)        \\
		=&\sum\limits_{T\in \mathcal{T}_{h}}\Vert \nabla\times \mathbf{u}-\mathbb{Q}_{h}\nabla \times \mathbf{u} \Vert_{T}^{2}-\sum\limits_{T\in \mathcal{T}_{h}}\gamma(h)h_{T}^{-1}\Vert \mathbf{Q}_{0}\mathbf{u}-\mathbf{Q}_{b}\mathbf{u} \Vert_{\partial T}^{2}.
	\end{flalign*}
	Combining Lemma $\ref{projection lower bound}$ with the trace inequality \cite{MR3223326}, there holds
	\begin{align*}
		&\sum\limits_{T\in \mathcal{T}_{h}}\Vert \nabla \times \mathbf{u}-\mathbb{Q}_{h}\nabla \times \mathbf{u} \Vert_{T}^{2}\gtrsim h^{2k},\\
		&\sum\limits_{T\in \mathcal{T}_{h}}\gamma(h)h_{T}^{-1}\Vert \mathbf{Q}_{0}\mathbf{u}-\mathbf{Q}_{b}\mathbf{u} \Vert_{\partial T}^{2}
		\lesssim \gamma(h)h^{2k}.
	\end{align*}
	Notice that $h$ is sufficiently small, $\gamma(h)\ll 1$, then
	$$\varepsilon_{h,\mathbf{u}}\gtrsim h^{2k}.$$
	It follows from Theorem $\ref{eigfunction error}$ that
	$$\lambda_{h}\Vert \mathbf{u}-\mathbf{u}_{h} \Vert^{2}\lesssim \gamma(h)^{-2} h^{2(k+1)}.$$
	Therefore, for sufficiently small $h$, the following inequality holds
	$$\varepsilon_{h,\mathbf{u}}\geq \lambda_{h}\Vert \mathbf{u}-\mathbf{u}_{h} \Vert^{2}.$$
	Then we complete the proof.
\end{proof}

From \cite[Theorem 2.4]{MR3919912}, we have the following theorem.
\begin{theorem}
	\label{lower bound}
	Let $(\lambda,\mathbf{u})$ be an eigenpair of $(\ref{weak form})$, $(\lambda_{h},\mathbf{u}_{h})$ be an eigenpair of $(\ref{WG-scheme})$, for sufficiently small $h$, we have
	$$\lambda\geq \lambda_{h}.$$
\end{theorem}

Combining Theorems \ref{eigenvalue error} and \ref{lower bound}, the asymptotic lower bound approximations of the exact eigenvalues are obtained.
%%%%%%%%%%%%%%%%%%%%%%%

\section{Numerical Experiments}
In this section, we present some numerical experiments.
\subsection{Square domain}
In first example, we consider the Maxwell eigenvalue problem $(\ref{weak form})$ on the square domain $\Omega=(0,\pi)^2$. Let $\gamma(h)=h^{0.1}$, $\varepsilon_r=1$ and $\mu_r=1$. The degree $k$ of the Wg finite element space is selected as 1. The first five exact eigenvalues is: $\lambda_{1}=1$, $\lambda_{2}=1$, $\lambda_{3}=2$, $\lambda_{4}=4$ and $\lambda_{5}=4$.

The numerical results are presented in Table \ref{table 1}. We can find that the first five eigenvalues all reach the optimal convergence orders and are the asymptotic lower bounds of the exact eigenvalues. The images and vectorgraph of the third eigenfunction are shown in Figures \ref{fig1}-\ref{fig3}.

\begin{table}[htbp]
	\label{table 1}
	\centering
	\caption{$\Omega=(0,\pi)^{2}$, $k=1$, $\gamma(h)=h^{0.1}$.}
	\renewcommand\arraystretch{1}
	\begin{tabular}{|c | c |c | c |c |}
		\hline
		$h$ & 1/8 & 1/16 & 1/32 & 1/64   \\ [0ex]
		\hline\hline
		$\lambda_{1}-\lambda_{1,h}$ & 1.3857e-01  &   4.0828e-02   &  1.1269e-02  &   3.0525e-03 \\
		\hline
		order &   & 1.7629 &   1.8572   & 1.8842 \\
		\hline
		$\lambda_{2}-\lambda_{2,h}$ & 6.4984e-02  &   1.8187e-02  &   4.9488e-03  &   1.3368e-03 \\
		\hline
		order &  & 1.8371 &    1.8777 &    1.8882\\
		\hline
		$\lambda_{3}-\lambda_{3,h}$ &  4.1608e-01    & 1.2117e-01   &  3.3223e-02    & 8.9721e-03\\
		\hline
		order &   & 1.7797  &   1.8668  &  1.8886 \\
		\hline
		$\lambda_{4}-\lambda_{4,h}$ & 1.3155e+00  &   4.4141e-01   &  1.2730e-01  &   3.4933e-02
		\\
		\hline
		order &   & 1.5754 &    1.7938 &    1.8655  \\
		\hline
		$\lambda_{5}-\lambda_{5,h}$ & 1.2921e+00    &  4.3273e-01   &  1.2657e-01    & 3.4878e-02
		\\
		\hline
		order &   &1.5782 &    1.7734 &   1.8595 \\
		\hline
	\end{tabular}
\end{table}

\begin{figure}[htbp]
	\centering
	\begin{minipage}[t]{0.48\textwidth}
		\centering
		\includegraphics[width=6cm]{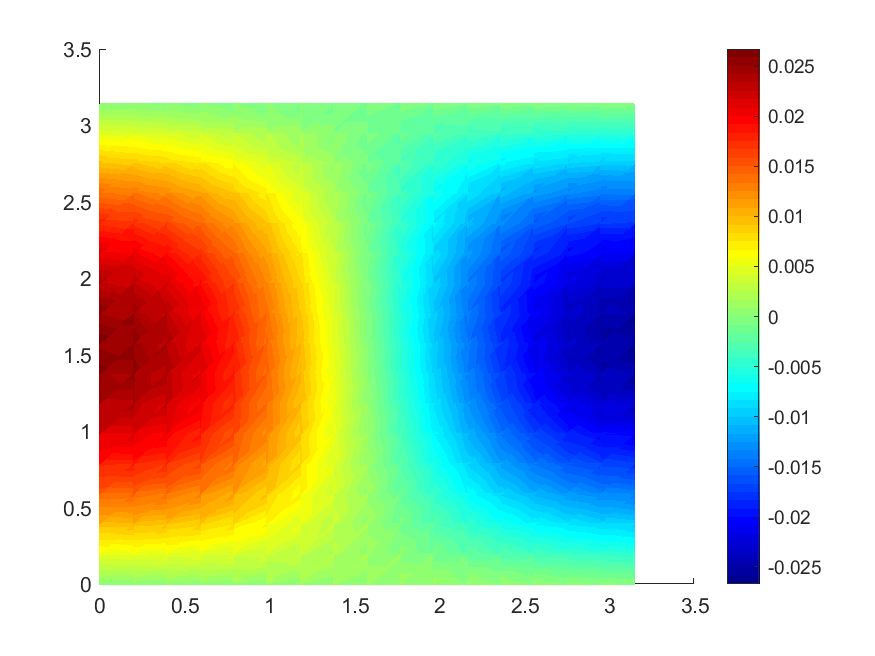}
		\caption{the first component of the third eigenfunction.}
			\label{fig1}
	\end{minipage}
	\begin{minipage}[t]{0.48\textwidth}
		\centering
		\includegraphics[width=6cm]{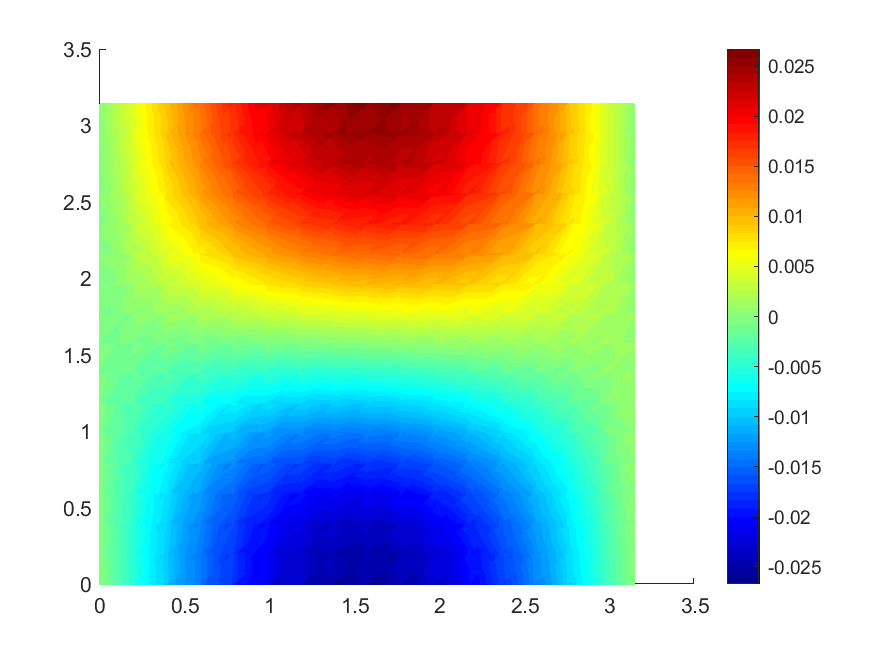}
		\caption{the second component of the third eigenfunction.}
			\label{fig2}
	\end{minipage}
	\begin{minipage}[t]{0.48\textwidth}
		\centering
		\includegraphics[width=6cm]{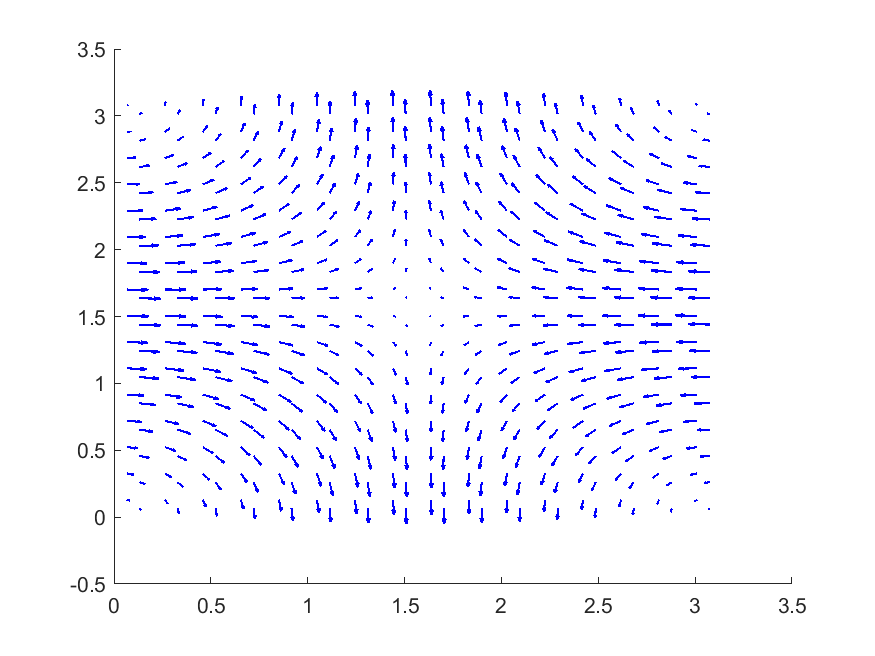}
		\caption{the vectorgraph of the third eigenfunction.}
		\label{fig3}
	\end{minipage}
\end{figure}

%-----------------------------------------------------------------------------------------------

\bibliographystyle{siam}
\bibliography{library}

\begin{thebibliography}{10}

\bibitem{MR1115240}
{\sc I.~Babu\v{s}ka and J.~Osborn}, {\em Eigenvalue problems}, in Handbook of
  numerical analysis, {V}ol. {II}, Handb. Numer. Anal., II, North-Holland,
  Amsterdam, 1991, pp.~641--787.

\bibitem{MR2652780}
{\sc D.~Boffi}, {\em Finite element approximation of eigenvalue problems}, Acta
  Numer., 19 (2010), pp.~1--120.

\bibitem{MR3097958}
{\sc D.~Boffi, F.~Brezzi, and M.~Fortin}, {\em Mixed finite element methods and
  applications}, vol.~44 of Springer Series in Computational Mathematics,
  Springer, Heidelberg, 2013.

\bibitem{MR3916956}
{\sc D.~Boffi and L.~Gastaldi}, {\em Adaptive finite element method for the
  {M}axwell eigenvalue problem}, SIAM J. Numer. Anal., 57 (2019), pp.~478--494.

\bibitem{MR3712172}
{\sc D.~Boffi, L.~Gastaldi, R.~Rodr\'{\i}guez, and I.~\v{S}ebestov\'{a}}, {\em
  Residual-based {\it a posteriori} error estimation for the {M}axwell's
  eigenvalue problem}, IMA J. Numer. Anal., 37 (2017), pp.~1710--1732.

\bibitem{MR3918688}
\leavevmode\vrule height 2pt depth -1.6pt width 23pt, {\em A posteriori error
  estimates for {M}axwell's eigenvalue problem}, J. Sci. Comput., 78 (2019),
  pp.~1250--1271.

\bibitem{MR4568428}
{\sc D.~Boffi, J.~Guzm\'{a}n, and M.~Neilan}, {\em Convergence of {L}agrange
  finite elements for the {M}axwell eigenvalue problem in two dimensions}, IMA
  J. Numer. Anal., 43 (2023), pp.~663--691.

\bibitem{cendes1986development}
{\sc Z.~Cendes, D.~Hudak, J.~Lee, and D.~Sun}, {\em Development of new methods
  for predicting the bistatic electromagnetic scattering from absorbing
  shapes}, RADC-TR R\&D,  (1986).

\bibitem{MR912525}
{\sc F.~Kikuchi}, {\em Mixed and penalty formulations for finite element
  analysis of an eigenvalue problem in electromagnetism}, in Proceedings of the
  first world congress on computational mechanics ({A}ustin, {T}ex., 1986),
  vol.~64, 1987, pp.~509--521.

\bibitem{kobelansky1986eliminating}
{\sc A.~Kobelansky and J.~Webb}, {\em Eliminating spurious modes in
  finite-element waveguide problems by using divergence-free fields},
  Electronics Letters, 11 (1986), pp.~569--570.

\bibitem{1573841}
{\sc J.-H. Lee, T.~Xiao, and Q.~Liu}, {\em A 3-d spectral-element method using
  mixed-order curl conforming vector basis functions for electromagnetic
  fields}, IEEE Transactions on Microwave Theory and Techniques, 54 (2006),
  pp.~437--444.

\bibitem{MR4566815}
{\sc Q.~Liang and X.~Xu}, {\em A two-level preconditioned {H}elmholtz subspace
  iterative method for {M}axwell eigenvalue problems}, SIAM J. Numer. Anal., 61
  (2023), pp.~642--674.

\bibitem{MR3120579}
{\sc Q.~Lin, H.~Xie, and J.~Xu}, {\em Lower bounds of the discretization error
  for piecewise polynomials}, Math. Comp., 83 (2014), pp.~1--13.

\bibitem{MR3371510}
{\sc N.~Liu, L.~Tob\'{o}n, Y.~Tang, and Q.~H. Liu}, {\em Mixed spectral element
  method for 2{D} {M}axwell's eigenvalue problem}, Commun. Comput. Phys., 17
  (2015), pp.~458--486.

\bibitem{PhysRevE.79.026705}
{\sc M.~Luo, Q.~H. Liu, and Z.~Li}, {\em Spectral element method for band
  structures of two-dimensional anisotropic photonic crystals}, Phys. Rev. E,
  79 (2009), p.~026705.

\bibitem{MR2059447}
{\sc P.~Monk}, {\em Finite element methods for {M}axwell's equations},
  Numerical Mathematics and Scientific Computation, Oxford University Press,
  New York, 2003.

\bibitem{MR3325251}
{\sc L.~Mu, J.~Wang, and X.~Ye}, {\em A weak {G}alerkin finite element method
  with polynomial reduction}, J. Comput. Appl. Math., 285 (2015), pp.~45--58.

\bibitem{MR3286455}
\leavevmode\vrule height 2pt depth -1.6pt width 23pt, {\em Weak {G}alerkin
  finite element methods on polytopal meshes}, Int. J. Numer. Anal. Model., 12
  (2015), pp.~31--53.

\bibitem{MR3394450}
{\sc L.~Mu, J.~Wang, X.~Ye, and S.~Zhang}, {\em A weak {G}alerkin finite
  element method for the {M}axwell equations}, J. Sci. Comput., 65 (2015),
  pp.~363--386.

\bibitem{MR592160}
{\sc J.-C. N\'{e}d\'{e}lec}, {\em Mixed finite elements in {${\bf R}^{3}$}},
  Numer. Math., 35 (1980), pp.~315--341.

\bibitem{MR864305}
\leavevmode\vrule height 2pt depth -1.6pt width 23pt, {\em A new family of
  mixed finite elements in {${\bf R}^3$}}, Numer. Math., 50 (1986), pp.~57--81.

\bibitem{rahman1984penalty}
{\sc B.~A. Rahman and J.~B. Davies}, {\em Penalty function improvement of
  waveguide solution by finite elements}, IEEE Transactions on Microwave Theory
  and Techniques, 32 (1984), pp.~922--928.

\bibitem{1969Finite}
{\sc P.~P. Silvester}, {\em Finite-element solution of homogeneous waveguide
  problems},  (1969).

\bibitem{MR3223326}
{\sc J.~Wang and X.~Ye}, {\em A weak {G}alerkin mixed finite element method for
  second order elliptic problems}, Math. Comp., 83 (2014), pp.~2101--2126.

\bibitem{MR3452926}
\leavevmode\vrule height 2pt depth -1.6pt width 23pt, {\em A weak {G}alerkin
  finite element method for the stokes equations}, Adv. Comput. Math., 42
  (2016), pp.~155--174.

\bibitem{winkler1984elimination}
{\sc J.~R. Winkler and J.~B. Davies}, {\em Elimination of spurious modes in
  finite element analysis}, Journal of Computational Physics, 56 (1984),
  pp.~1--14.

\bibitem{MR3919912}
{\sc Q.~Zhai, H.~Xie, R.~Zhang, and Z.~Zhang}, {\em The weak {G}alerkin method
  for elliptic eigenvalue problems}, Commun. Comput. Phys., 26 (2019),
  pp.~160--191.

\end{thebibliography}

\end{document}